\documentclass[11pt]{amsart}

\newtheorem{theorem}{Theorem}[section]

\newtheorem{corollary}[theorem]{Corollary}

\newtheorem{definition}[theorem]{Definition}
\newtheorem{lemma}[theorem]{Lemma}

\newtheorem{problem}[theorem]{Problem}
\newtheorem{proposition}[theorem]{Proposition}

\usepackage[colorlinks, bookmarks=true]{hyperref}
\usepackage{color,graphicx,shortvrb}
\usepackage[active]{srcltx} 

\def\J#1#2#3{ \left\{ #1,#2,#3 \right\} }
\def\RR{{\mathbb{R}}}

\def\11{\textbf{$1$}}

\usepackage{enumerate}
\usepackage{amsmath}
\usepackage{amssymb}
\usepackage[latin 1]{inputenc}

\title[Derivable linear maps on C$^*$-algebras]{Linear maps on C$^*$-algebras which are derivations or triple derivations at a point}

\author[Ben Ali Essaleh]{Ahlem Ben Ali Essaleh}
\email{ahlem.benalisaleh@gmail.com}
\address{Faculte des Sciences de Monastir, Département de Mathématiques, Avenue de L'environnement, 5019 Monastir, Tunisia}

\author[Peralta]{Antonio M. Peralta}
\email{aperalta@ugr.es}
\address{Departamento de An{\'a}lisis Matem{\'a}tico, Facultad de
Ciencias, Universidad de Granada, 18071 Granada, Spain.}



\begin{document}

\maketitle

\begin{abstract} We prove new results on generalized derivations on C$^*$-algebras. By considering  the triple product  $\{a,b,c\} =2^{-1} (a b^* c + c b^* a)$, we introduce the study of linear maps which are triple derivations or triple homomorphisms at a point. We prove that a continuous linear $T$ map on a unital C$^*$-algebra is a generalized derivation whenever it is a triple derivation at the unit element. If we additionally assume $T(1)=0,$ then $T$ is a $^*$-derivation and a triple derivation. Furthermore, a continuous linear map on a unital C$^*$-algebra which is a triple derivation at the unit element is a triple derivation. Similar conclusions are obtained for continuous linear maps which are derivations or triple derivations at zero.

We also give an automatic continuity result, that is, we show that generalized derivations on a von Neumann algebra and linear maps on a von Neumann algebra which are derivations or triple derivations at zero are all continuous even if not assumed a priori to be so.
\end{abstract}

\section{Introduction}

Automorphisms and derivations on Banach algebras are among the most intensively studied classes of operators. Recent studies are concerned with the question of finding weaker conditions to characterize these maps. One of the most fruitful lines studies maps which are derivations or automorphisms at a certain point. More concretely, a linear map $S$ from a Banach algebra $A$ to a Banach $A$-bimodule $X$ is said to be a \emph{derivation} at a point $z\in A$ if the identity \begin{equation}\label{eq derivation} S(ab ) = S(a) b + a S(b),
\end{equation} holds for every $a,b\in A$ with $a b =z$. In the literature a linear map which is a derivation at a point $z$ is also called \emph{derivable} at $z$. Clearly, a linear map $D$ from $A$ into $X$ is a derivation if and only if it is a derivation at every point of $A$. We can similarly define linear maps which are Jordan derivations or generalized derivations at a point (see subsection \ref{subsect1} for detailed definitions).\smallskip

Following the terminology set by J. Alaminos, M. Bresar, J. Extremera, and A. Villena in \cite[\S 4]{AlBreExVill09} and J. Li and Z. Pan in \cite{LiPan}, we shall say that a linear operator $G$ from a Banach algebra $A$ into a Banach $A$-bimodule $X$ is a \emph{generalized derivation} if there exists $\xi\in  X^{**}$ satisfying $$G(ab) = G(a)  b + a  G(b) - a \xi b \hbox{ ($a, b \in A$).}$$ Every derivation is a generalized derivation, however there exist generalized derivations which are not derivation. This notion is very useful when characterizing (generalized) derivations in terms of annihilation of certain products of orthogonal elements (see, for example, Theorem 2.11 in \cite[\S 2]{AyuKudPe2014}).\smallskip

The first results on linear maps that are derivable at zero appear in \cite[Subsection 4.2]{Bre07} and \cite[Theorem 2]{ChebKeLee}, where they were related to generalized derivations. In \cite[Theorem 4]{JingLuLi2002} W. Jing, S.J. Lu, and P.T. Li prove that the implication $$\delta \hbox{ is a derivation at zero } \Rightarrow \delta \hbox{ is a generalized derivation},$$ holds for every continuous linear map $\delta$ on a von Neumann algebra. If, under the above hypothesis $\delta (1)=0$, then $\delta$ is a derivation. We shall prove in Corollary \ref{c JingLuLi without continuity} that the hypothesis concerning continuity can be relaxed.\smallskip

W. Jing proves in \cite[Theorems 2.2 and 2.6]{Jing} the following result: for an infinite dimensional Hilbert space $H$, a linear map $\delta: B(H) \to B(H)$ which is a generalized Jordan derivation at zero, or at 1, is a generalized derivation. We observe that, in the latter result, $\delta$ is not assumed to be continuous.\smallskip

More related results read as follow. Let $X$ be a Banach $A$-bimodule over a Banach algebra $A$. In 2009, F. Lu establishes that a linear map $\delta : A\to X$ is a derivation whenever it is continuous and a derivation at an element which is left (or right) invertible (see \cite{Lu2009}). Is is further shown that $\delta$ is a derivation if it continuous and a derivation at an idempotent $e$ in $A$ such that for $x\in X$ the condition $e A (1-e) X =\{0\}$ implies $(1-e) X=\{0\}$ and the condition $XeA(1-e)=\{0\}$ gives $X e =0$. Here the linear map is assumed to be continuous.\smallskip

Concerning our goals, J. Zhu, Ch. Xiong, and P. Li prove in \cite{ZhuXiLi} a significant result showing that, for a Hilbert space $H$, a linear map $\delta :B(H)\to B(H)$ is a derivation if and only if it is a derivation at a non-zero point in $B(H)$. It is further shown that a linear map which is a derivation at zero need not be a derivation (for example, the identity mapping on $B(H)$ is a derivation at zero but it is not a derivation).\smallskip

We refer to \cite{Houqi, JingLuLi2002, ZhangHouQi2014, ZhangHouQi2014b, Zhu2007, ZhuXio2002, ZhuXio2005, ZhuXio2007, ZhuXion2008} and \cite{ZhuZhao2013} for additional results on linear or additive maps on JSL algebras, finite CSL algebras, nest algebras or standard operator algebras.\smallskip

In the present note we continue with the study of those linear maps which are derivable at zero. We shall introduce a new point of view by exploiting those properties of a C$^*$-algebra $A$ which are related to the ternary product defined by \begin{equation}\label{eq C*-triple product}\{a,b,c\}=\frac12 (a b^* c + cb^* a) \ \ \ (a,b,c\in A).
\end{equation}

Every C$^*$-algebra $A$ is a JB$^*$-triple (in the sense of \cite{Kaup83}) with respect to the triple product defined in \eqref{eq C*-triple product}. This is the natural triple product appearing in the study of J$^*$-algebras by L.A. Harris \cite{Harr74,Harris81} and the \emph{ternary rings of operators} (TRO's) in the sense of D.P. Blecher and M. Neal in \cite{BleNe} and M. Neal and B. Russo in \cite{NealRusso}.\smallskip

A linear map $T$ between C$^*$-algebras preserving the previous triple product is called a \emph{triple homomorphism}. A triple derivation on a C$^*$-algebra $A$ is a linear map $\delta : A\to A$ satisfying the generalized Leibnitz's rule $$\delta \{a,b,c\} = \{\delta(a),b,c\}+ \{a,\delta(b),c\}+\{a,b,\delta(c)\},$$ for all $a,b,c\in A$.\smallskip

We recall that a $^*$-derivation on a C$^*$-algebra $A$ is a derivation $D:A\to A$ satisfying $D(a)^*= D(a^*)$ for all $a\in A$.  Examples of derivations on $A$ be given by fixing $z\in A$ and defining $D_z : A\to A$ as the linear map defined by $D_z (a) = [z,a] = z a - az$. It is known that every $^*$-derivation on a C$^*$-algebra is a triple derivation in the above sense. It is further known the existence of derivations on $A$ which are not triple derivations (compare \cite[Comments after Lemma 3]{BurFerGarPe2014}). \smallskip

On the other hand, for each $a$ in a C$^*$-algebra $A$, the mapping $\delta_a (x) := i  \{a,a,x\}$ is a triple derivation on $A$, however, $ i (a^* a + a a^*)= 2\delta_a (1) =\delta_a (1)=  \frac{i}{2} (a^* a + a a^*)$ if and only if $a=0$, and thus $\delta_a$ is not an associative derivation on $A$ for every $a\neq 0$.\smallskip

In a recent paper M.J. Burgos, J.Cabello-S{\'a}nchez and the second author of this note explore those linear maps between C$^*$-algebras which are $^*$-homomorphisms at certain points of the domain, for example, at the unit element and at zero (see the introduction of section \ref{sec: triple hom at a point} for more details).\smallskip

In this paper we widen the scope by introducing linear maps which are triple derivations or triple homomorphism at a certain point. Our study will be conducted around the next two notions.

\begin{definition}\label{def triple derivation at a point}
Let $T:A\to A$ be a linear map on a C$^*$-algebra, and let $z$ be an element in $A.$  We shall say that $T$ is a triple derivation at $z$ if $z=\{a,b,c\}$ in $A$ implies that $T(z)=\{T(a),b,c\}+\{a,T(b),c\}+\{a,b,T(c)\}.$
\end{definition}

The set of all linear maps on $A$ which are triple derivable at an element $z\in A$ is a subspace of the space $L(A)$ of all linear operators on $A$.

\begin{definition}\label{def triple homomorphism at a point}
Let $T:A\to B$ be a linear map between C$^*$-algebras, and let $z$ be an element in $A.$  We shall say that $T$ is a triple homomorphism at $z$ if $z=\{a,b,c\}$ in $A$ implies that $\{T(a),T(b),T(c)\}=T(z).$
\end{definition}

Let $T$ be a continuous linear map on a unital C$^*$-algebra. In Theorem \ref{t triple derivation at 1 are generalized derivations} we prove that $T$ being a triple derivation at the unit implies that $T$ is a generalized derivation. If we also assume that $T(1)=0,$
then $T$ is a $^*$-derivation and a triple derivation (see Proposition \ref{symmetric}). Among the consequences, we establish that a continuous linear map on a unital C$^*$-algebra which is a triple derivation at the unit element is a triple derivation (see Corollary \ref{c derivable at 1 are triple derivations}).\smallskip

When we study linear maps which are triple derivation at zero, our conclusions are stronger. We begin with an extension of \cite[Theorem 4]{JingLuLi2002} to the setting of unital C$^*$-algebras. We show that a continuous linear map $T$ on a C$^*$-algebra is a generalized derivation whenever it is a derivation or a triple derivation at zero (see Theorem \ref{t continuous triple der at zero are g der}). Moreover, a bounded linear map $T$ on a C$^*$-algebra $A$ which is a triple derivation at zero with $T(1)=0$ is a $^*$-derivation, and hence a triple derivation (compare Corollary \ref{c bl triple derivable at zero with T1=0}). We further show that a bounded linear map on a unital C$^*$-algebra $A$ which is a triple derivation at zero and satisfies $T(1)^*=-T(1)$ is a triple derivation (see Corollary \ref{c T1 skew}).\smallskip

For linear maps whose domain is a von Neumann algebra the continuity assumptions can be dropped for certain maps. More concretely, generalized derivations on a von Neumann algebra, linear maps on a von Neumann algebra which are derivations (respectively, triple derivations) at zero are all continuous (see Corollary \ref{c automatic cont gen der and triple der at zero on von Neumann}). Several characterizations of generalized derivations on von Neumann algebras are established in Corollary \ref{c non continuous generalized derivation} without assuming continuity. In this particular setting, some hypothesis in \cite[Theorem 4]{JingLuLi2002} and \cite{Lu2009} can be relaxed.\smallskip

In section \ref{sec: triple hom at a point} we study continuous linear maps on C$^*$-algebras which are triple homomorphisms at zero or at the unit element. Let $T:A\to B$ be a continuous linear map between C$^*$-algebras, where $A$ is unital. We prove in Theorem \ref{theorem3131} that if $T$ is a triple homomorphism at the unit of $A,$ then $T$ is a triple homomorphism. Furthermore, $T(1)$ is a partial isometry and $T: A \to B_2 (T(1))$ is a Jordan $^*$-homomorphism.\smallskip

For triple homomorphisms at zero, we rediscover the orthogonality preserving operators. More concretely, let $T :A \to B$ be a bounded linear map between two C$^*$-algebras. We shall revisit the main results in \cite{BurFerGarMarPe08} to show that $T$ is orthogonality preserving if, and only if, $T$ preserves zero-triple-products (i.e. $\{a,b,c\}=0$ in $A$ implies $\{T(a),T(b),T(c)\}=0$ in $B$) if, and only if, $T$ a triple homomorphism at zero.

\subsection{Basic background and definitions}\label{subsect1} \ \ \vspace*{1mm}

The class of C$^*$-algebras admits a Jordan analogous in the wider category of JB$^*$-algebras. More concretely, a real (resp., complex) \emph{Jordan algebra} is an algebra $\mathcal{J}$ over the real (resp., complex) field whose product is commutative (but, in general, non-associative) and satisfies the \emph{Jordan identity}:\begin{equation}\label{eq Jordan idenity algebra} (a \circ b)\circ a^2 = a\circ (b \circ a^2).
\end{equation} A JB$^*$-algebra is a complex Jordan algebra $\mathcal{J}$ which is also a Banach space and admits an isometric algebra involution $^*$ satisfying $\| a\circ b\| \leq \|a\| \ \|b\|,$ and $$\|\J a{a^*}a \|= \|a\|^3,$$ for all $a,b\in \mathcal{J}$, where $\J a{a^*}a =2 (a\circ a^*) \circ a - a^2 \circ a^*$. Every C$^*$-algebra is a JB$^*$-algebra with respect to its natural norm and involution and the Jordan product given by $a\circ b = \frac12( a b +b a)$. The self-adjoint part $\mathcal{J}_{sa}$ of a JB$^*$-algebra $\mathcal{J}$ is a real Jordan Banach algebra which satisfies $$\|a\|^2 = \|a^2\| \hbox{ and } \|a^2\|\leq \|a^2+b^2\|,$$ for every $a,b\in \mathcal{J}_{sa}.$ These axioms provide the precise definition of JB-algebras. A \emph{JBW$^*$-algebra} (resp., a \emph{JBW-algebra}) is a JB$^*$-algebra (resp., a JB-algebra) which is also a dual Banach space. The bidual of every JB$^*$-algebra is a JBW$^*$-algebra with a Jordan product and involution extending the original ones. The reader is referred to the monograph \cite{Hanche} for the basic background on JB$^*$- and JB-algebras.\smallskip

Let $\mathcal{B}$ be a JB$^*$-subalgebra of a JB$^*$-algebra $\mathcal{J}$. Accordingly to the notation in \cite{AlBreExVill09,BurFerGarPe2014,BurFerPe2014} a linear mapping $G: \mathcal{B}\to \mathcal{J}$ will be called a \emph{generalized Jordan derivation} if there exists $\xi\in \mathcal{J}^{**}$ satisfying \begin{equation}\label{eq generalized Jordan derivation} G (a\circ b) = G(a)\circ b + a\circ G(b) - U_{a,b} (\xi ),
 \end{equation} for all $a,b$ in $\mathcal{J}$, where $U_{a,b} (x) := (a\circ x) \circ b + (b\circ x)\circ a - (a\circ b) \circ x$ ($x\in \mathcal{J}$). We shall write $U_a$ instead of $U_{a,a}$. If $\mathcal{B}$ is unital, every generalized Jordan derivation $G : \mathcal{B}\to \mathcal{J}$ with $G(1) =0$ is a Jordan derivation. Jordan derivations are generalized Jordan derivations.\smallskip

\section{Triple derivations at fixed point of a C$^*$-algebra}

In this section we shall study linear maps between C$^*$-algebras which are triple derivations at a fixed point of the domain. There are two remarkable elements that every study should consider in a first stage, we refer to zero and the unit element of a C$^*$-algebra. We shall show later that linear maps between C$^*$-algebras which are triple derivations at zero or at the unit element are intrinsically related to generalized derivations.\smallskip

Let $T: A \to X$ be a bounded linear operator from a C$^*$-algebra into an essential Banach $A$-bimodule. It is proved in \cite[Theorem 4.5]{AlBreExVill09} and \cite[Proposition 4.3]{BurFerPe2014} (see also \cite[Theorem 2.11]{AyuKudPe2014}) that $T$ is a generalized derivation if and only if one of the next statements holds:\begin{enumerate}[$(a)$]
\item $T$ is a generalized derivation;
\item $a T(b) c = 0$, whenever $ab=bc=0$ in $A$;
\item $a T(b) c = 0$, whenever $ab=bc=0$ in $A_{sa}$.
\end{enumerate}

When in the above statement $X$ coincides with $A$ or with any C$^*$-algebra containing $A$ as a C$^*$-subalgebra with the same unit, the above equivalent statements admit another reformulation which is more interesting for our purposes. We shall isolate here an equivalence which was germinally contained in the proof of \cite[Theorem 2.8]{EssaPeRa16}. More concretely, each statement in $(a)$-$(c)$ is equivalent to any of the following:\label{eq reformulations of g der}
\begin{enumerate}[$(d)$]
\item[$(d)$] $a T(b) c + c T(b) a = 0$, whenever $ab=bc=0$ in $A_{sa}$;
\item[$(e)$] $a T(b) a = 0$, whenever $ab=0$ in $A_{sa}$.
\item[$(f)$] For each $b$ in $A_{sa}$ we have $(1-r(b)) T(b) (1-r(b))=0$ in $B^{**}$, where $r(b)$ denotes the range projection of $b$ in $A^{**}$.
\end{enumerate}

For the proof we observe that $(c)\Rightarrow (d)$ and $(d)\Rightarrow (e)$ are clear. We shall prove $(e)\Rightarrow (f)$. Suppose  $a T(b) a = 0$, whenever $ab=0$ in $A_{sa}$. We shall focus on the commutative C$^*$-subalgebra $A_b$ generated by $b$.\smallskip

It is known from the Gelfand theory that $A_b \cong C_0(\sigma(b))$, where $\sigma (b)\subseteq [-\|b\|,\|b\|]$ denotes the spectrum of $b$ and $C_0(\sigma(b))$ the C$^*$-algebra of all continuous functions on $\sigma(b)$ vanishing at zero. For each natural $n$, let $p_n$ denote the projection in $A_b^{**}\subseteq A^{**}$ corresponding to the characteristic function of the set $([-\|b\|, -\frac1n]\cup  [\frac1n,\|b\|])\cap \sigma(b).$ Let us also pick a function $b_n\in A_b$ such that $b_n p_n = p_n b_n = b_n=b_n^*$ and $\|b_n-b\|< \frac1n$. Clearly, $(p_n)$ converges to $r(b)$, the range projection of $b$ in the strong$^*$-topology of $A^{**}$ (see \cite[\S 1.8]{S}). Let us take $z \in ((1-p_n) A^{**} (1-p_n))\cap A_{sa}$. Since $b_n z = 0$, it follows from the hypothesis that $z T(b_n) z =0$.\smallskip

On the other hand, it is known that $p_n$ is a closed projection in $A_b^{**}\subseteq A^{**}$ in the sense employed in \cite[Definition III.6.19]{Tak}. It is known that, under these circumstances, $((1-p_n) A^{**} (1-p_n))\cap A_{sa}$ is weak$^*$-dense in $(1-p_n) A^{**} (1-p_n)$ (compare \cite[Proposition 3.11.9]{Ped}). By Kaplansky density theorem \cite[Theorem 1.9.1]{S},  we can find a bounded net $(z_\lambda)$ in $((1-p_n) A^{**} (1-p_n))\cap A_{sa}$ converging to $1-p_n$ in the strong$^*$-topology of $A^{**}$. We have seen above that $z_\lambda T(b_n) z_\lambda =0 $ for all $\lambda$. Since the product of $A$ is jointly strong$^*$-continuous (cf. \cite[Proposition 1.8.12]{S}), we deduce that $(1-p_n) T(b_n) (1-p_n) =0 $ for all natural $n$. Since $(1-p_n)\to 1-r(b)$ in the strong$^*$-topology and $T(b_n) \to T(b)$ in norm, we have $(1-r(b)) T(b) (1-r(b))=0$ in $A^{**}$.  \smallskip

$(f)\Rightarrow (c)$ We take $a,b,c\in A_{sa}$ with $a b = bc=0$. We can easily see that $a = a (1-r(b))$ and $c = (1-r(b)) c$. Therefore, by assumptions, $a T(b) c = a (1-r(b)) T(b) (1-r(b)) c = 0$, which finishes the proof.\medskip\medskip

\subsection{Triple derivations at the unit element of a C$^*$-algebra}\  \vspace{2mm}

Along the rest of this subsection, the symbol $A$ will denote a C$^*$-subalgebra of unital C$^*$-algebra $B$, and we shall assume that $A$ contains the unit of $B$.\smallskip

Continuous linear maps $T: A\to B$ which are derivations at 1 are derivations. This problem has been already studied in the literature, at least for continuous linear maps (see \cite[Theorem 2.1 or Corollary 2.3]{Lu2009}). Actually the next result follows from the just quoted reference and \cite[Theorem 6.3]{John96}.

\begin{proposition}\label{p cont derivation at 1}{\rm(\cite[Theorem 2.1 or Corollary 2.3]{Lu2009}, \cite[Theorem 6.3]{John96})} Let $A$ be a unital C$^*$-algebra, and $X$ be a unital Banach $A$-bimodule. Suppose $T: A\to B$ is a continuous linear map which is a derivation at the unit element. Then $T$ is a derivation.$\hfill\Box$
\end{proposition}

A common property of triple derivations and local triple derivations on C$^*$-algebras is that they map the unit of the domain C$^*$-algebra into a skew symmetric element (cf. \cite[proof of Lemma 1]{HoMarPeRu}, \cite[Lemma 3.4]{HoPeRu} or \cite[Lemma 2.1]{KuOikPeRu14}). This good behavior is also true for linear maps which are derivations at the unit element.

\begin{lemma}\label{lemma33} Let $T:A\to B$ be a triple derivation at the unit of $A.$
Then the following statements hold:
\begin{enumerate}[$(a)$]
    \item $T(1)^*=-T(1);$
\item The identity $T(p)=T(p)p+pT(p)-pT(1)p,$ holds for every projection $p$ in $A.$
\end{enumerate}
\end{lemma}

\begin{proof}$(a)$ Since $1=\{1,1,1\},$ by assumptions, we have
$$T(1)=T(\{1,1,1\})= 2\{T(1),1,1\}+\{1,T(1),1\}=2T(1)+T(1)^*,$$
which proves the statement. \smallskip

$(b)$ Let $p\in A$ be a projection. The identity $\{(1-2p),1,(1-2p)\}=1$ and the hypothesis prove that
$$T(1)=T(\{(1-2p),1,(1-2p)\})$$ $$ =2\{T(1-2p),1,(1-2p)\}+\{(1-2p),T(1),(1-2p)\}$$
$$=(T(1)-2T(p))(1-2p)+(1-2p)(T(1)-2T(p))+(1-2p)T(1)^*(1-2p)$$
$$=2T(1)-4T(p)+4T(p)p+4pT(p)-2T(1)p$$
$$-2pT(1)-2T(1)^*p-2pT(1)^*+4pT(1)^*p+T(1)^*$$
$$=\hbox{(by $(1)$)} = T(1)- 4T(p) + 4 T(p) p + 4 p T(p) -4 p T(1) p.$$
This implies that $T(p)=T(p)p+pT(p)-pT(1)p.$
\end{proof}

There exist C$^*$-algebras containing no non-zero projections. For this reason, we need to deal with unitaries.

\begin{theorem}\label{t triple derivation at 1 are generalized derivations}
Let $T:A\to B$ be a continuous linear map which is a triple derivation at the unit of $A.$
Then $T$ is a generalized derivation.
\end{theorem}

\begin{proof} Let us take $a\in A_{sa}.$ Since, for each $t\in \RR,$ $e^{it a}$ is a unitary element in $A$ and $1=\{e^{i t a},\;1,e^{-i t a}\},$ we deduce that $$T(1)= \{T(e^{i t a}),\;1,e^{-i t a}\}+\{e^{i t a},\;T(1),e^{-i t a}\}+\{e^{i t a},\;1,T(e^{-i t a})\}.$$
Taking the first derivative in $t$ we get $$0=\{T(ae^{i t a}),\;1,e^{-i t a}\}-\{T(e^{i t a}),\;1,ae^{-i t a}\}+\{ae^{i t a},\;T(1),e^{-i t a}\}$$
$$-\{e^{i t a},\;T(1),ae^{-i t a}\}+
\{ae^{i t a},\;1,T(e^{-i t a})\}-\{e^{i t a},\;1,T(ae^{-i t a})\}$$
for every $t\in \RR.$ Taking a new derivative at $t=0$ in the last equality, we get
$$0=2 \{T(a^2 ),\;1,1\} - 4 \{T(a),\;1,a \} + 2 \{T(1),\;1,a^2\} + 2 \{a^2 ,\;T(1),1 \}$$ $$ - 2 \{a ,\;T(1), a\},$$ or equivalently,
$$2T(a^2)=2T(a)a+2aT(a)+2aT(1)^*a-T(1)a^2-a^2T(1)-T(1)^*a^2-a^2T(1)^*.$$
Lemma \ref{lemma33}$(a)$ assures that $T(1)^*=-T(1),$ which implies that
\begin{equation}\label{eq 1 theorem 1}
 T(a^2)=T(a)a+aT(a)-aT(1)a,
\end{equation}
for every $a$ in $A_{sa}.$\smallskip

Finally, let us take $a,b,c\in A_{sa}$ with $a b=0= bc$. If we write $b= b^+ - b^-$ with $b^+ b^- = 0$ and $b^{\sigma} \geq 0$ for all $\sigma\in \{ \pm \}$. Find $d^{\sigma} \geq 0$ in $A$ such that $(d^{\sigma})^2 = b^{\sigma}$ ($\sigma=\pm$). It is not hard to check (for example, by applying the orthogonality of the corresponding range projections) that $a d^{\sigma} = d^{\sigma} c=0$ for $\sigma = \pm$. Now applying \eqref{eq 1 theorem 1} we get $$ a T(b) c = a T(b^+) c -  a T(b^-) c $$ $$= a( T(d^+) d^+ +d^+ T(d^+) )c - a( T(d^-) d^- + d^- T(d^-) )c =0.$$ We deduce from \cite[Theorem 2.11]{AyuKudPe2014} that $T$ is a generalized derivation.
\end{proof}

There exists generalized derivations which are not triple derivations at 1. For example, let $a$ be a non-zero symmetric element in $A$ and define $T(x)= a x$ ($\forall x\in A$). Then $T( xy ) = a x y = T(x) y + x T(y) - x T(1) y = a x y +x a y - x a y,$ for all $x,y\in A$, which assures that $T$ is a generalized derivation. However, $T(1) = a \in A_{sa}\backslash \{0\}$ together with Lemma \ref{lemma33} assure that $T$ is not a triple derivation at $1$.\smallskip

An appropriate change in the arguments given in the above theorem provide additional information when $T(1)=0$.

\begin{proposition}\label{symmetric} Let  $T:A\to A$ be a continuous linear map which is a triple derivation at 1 with $T(1)=0.$
Then $T$ is a $^*$-derivation and a triple derivation.
\end{proposition}

\begin{proof} As in the above proof, let us take $a\in A_{sa}.$ Since, for each $t\in \RR,$ $e^{it a}$ is a unitary element in $A$ and $1=\{e^{i t a},e^{i t a}, 1\},$ we deduce that $$0=T(1)= \{T(e^{i t a}),e^{i t a}, 1\}+\{e^{i t a}, T(e^{i t a}), 1\}+\{e^{i t a}, e^{i t a}, T(1) \}, $$ that is, $$ 0= T(e^{i t a})\circ e^{-i t a} + e^{i t a}\circ T(e^{i t a})^*.$$ Taking derivatives at $t=0$ we get $$0= T(a) - T(1) \circ a + a \circ T(1)^* - T(a)^*,$$ which proves that $T(a) = T(a)^*$ for all $a\in A_{sa}$, and thus $T$ is a symmetric map.\smallskip

Finally, by Theorem \ref{t triple derivation at 1 are generalized derivations}, $T$ is a generalized derivation. Furthermore, since $T(1)=0$ and $T$ is a symmetric operator, we deduce that $T$ is a $^*$-derivation and a triple derivation as well.
\end{proof}

\begin{corollary}\label{c derivable at 1 are triple derivations} Let $T:A\to A$ be a continuous linear map on a unital C$^*$-algebra which is a triple derivation at 1.  Then $T$ is a triple derivation.
\end{corollary}

\begin{proof} Since, by Lemma \ref{lemma33}, $T(1)^* =-T(1),$ it is known that the mapping $\delta(T(1),1):A \to B,$ $\delta(T(1),1) (x) = \{T(1),1,x\}- \{1,T(1),x\}$ is a triple derivation (compare \cite[Proof of Lemma 1]{HoMarPeRu}). Since the linear combination of linear maps which are triple derivations at $1$ is a triple derivation at $1$, the mapping $\tilde{T}=T-\frac{1}{2}\delta(T(1),1)$ is a triple derivation at 1, and $\tilde{T} (1) = 0$. Applying Proposition \ref{symmetric}, we derive that $\tilde{T}$ is a $^*$-derivation. Therefore $T= \tilde{T} +\frac{1}{2}\delta(T(1),1)$ is a triple derivation.
\end{proof}

\begin{problem}
If $T:A\to B$ is a triple derivation at the unit of a unital C$^*$-algebra $A$,  is $T$ continuous?
\end{problem}

\subsection{Triple derivations at zero}\ \vspace{1mm}

We begin this subsection exploring the basic properties of linear maps which are derivations at zero.

\begin{lemma}\label{lemma1.2 associative der} Suppose $A$ is a C$^*$-subalgebra of a C$^*$-algebra $B$, and let $T:A\to B$ be a linear map which is a derivation at zero. Then
$$a T(b) c=0,\quad \forall \;a,\;b,\;c\in A \text{ with } a b= b c=0.$$
\end{lemma}

\begin{proof} Suppose $a,b,c\in A$ satisfy the hypothesis of the lemma. Since $T$ is a derivation at zero we have $aT(b) c = (T(ab)-T(a) b) ) c =0.$
\end{proof}

Let us observe that under the hypothesis of the above lemma, we are not in a position to apply \cite[Theorem 4.5]{AlBreExVill09} and \cite[Proposition 4.3]{BurFerPe2014} (see also \cite[Theorem 2.11]{AyuKudPe2014}) and the reformulations we have reviewed in page \pageref{eq reformulations of g der} because $T$ is not assumed to be continuous. We shall see later that continuity can be relaxed when the domain is a von Neumann algebra.\smallskip

Let $a,\;b\in A,$ we recall that $a$ and $b$ are orthogonal (written $a\perp b$) if and only if $ab^*=b^*a=0.$

\begin{lemma}\label{lemma1.2} Suppose $A$ is a C$^*$-subalgebra of a C$^*$-algebra $B$, and
let $T:A\to B$ be a linear map which is a triple derivation at zero. Then
$$\{a,T(b),c\}=0,\quad \forall \;a,\;b,\;c\in A \text{ with } a\perp b \perp c.$$
\end{lemma}

\begin{proof} Let us take $a,\;b,\;c\in A,$ satisfying $a\perp b \perp c.$ Since $\{a,b,c\}=0,$ it follows from the hypothesis that
$$ 0=\{T(a),b,c\}+\{a,T(b),c\}+\{a,b,T(c)\},$$
which proves the statement because $\{T(a),b,c\}= \{a,b,T(c)\}=0$.
\end{proof}

We can now apply the reformulations of being a generalized derivation proved in page \pageref{eq reformulations of g der}. Let us recall that as observed by J. Zhu, Ch. Xiong, and P. Li in \cite{ZhuXiLi} linear maps which are derivations at zero need not be derivations. We shall see next that continuous linear maps which are derivations at zero are always generalized derivations.

\begin{theorem}\label{t continuous triple der at zero are g der} Let $A$ be a C$^*$-subalgebra of a unital C$^*$-algebra $B$. Let $T:A\to B$ be a bounded linear map. If $T$ is a derivation at zero or a triple derivation at zero, then $T$ is a generalized derivation.
\end{theorem}

\begin{proof} If $T$ is a triple derivation at zero, by Lemma \ref{lemma1.2}, given $a,b,c\in A_{sa}$ with $a b = b c =0$ we have $$0=2 \{a,T(b),c\}= a T(b)^* c + c T(b)^* a ,$$ or equivalently, $$0= a T(b) c + c T(b) a .$$ Lemma \ref{lemma1.2 associative der} assures that a similar conclusion holds when $T$ is a derivation at zero.  It follows from the equivalence $(d)\Leftrightarrow (a)$ in page \pageref{eq reformulations of g der} that $T$ is a generalized derivation.
\end{proof}

We observe that Theorem \ref{t continuous triple der at zero are g der} above extends \cite[Theorem 4]{JingLuLi2002} to the setting of unital C$^*$-algebras.

\begin{corollary}\label{c bl triple derivable at zero with T1=0} Suppose $A$ is a unital C$^*$-algebra. Let $T:A\to A$ be a bounded linear map which is a triple derivation at zero with $T(1)=0.$ Then $T$ is a $^*$-derivation, and hence a triple derivation.
\end{corollary}

\begin{proof} Since $T(1)=0$, the previous Theorem \ref{t continuous triple der at zero are g der} assures that $T$ is a derivation. We shall next show that $T$ is a symmetric mapping. It is well known that the bitransposed $T^{**}: A^{**}\to B^{**}$ is a derivation too (see for example \cite[Lemma 4.1.4]{S}). To avoid confusion with the natural involution on $A$, we shall denote $T^{**}$ by $T$. \smallskip

Fix $b\in A_{sa}$ with range projection $r(b) \in A^{**}$. Applying the same arguments given in the proof of  $(e)\Rightarrow (f)$ in page \pageref{eq reformulations of g der}, we can find sequences $(p_n)\subseteq A^{**}$ and $(b_n)\in A_b \cong C_0(\sigma(b))$ such that $\|b_n-b\|\to 0$, $p_n$ is a closed projection in $A^{**}$ for every $n$, $b_n p_n = b_n$, and for each natural $n$, there exists a bounded net $(z_\lambda)$ in $((1-p_n) A^{**} (1-p_n))\cap A_{sa}$ converging to $1-p_n$ in the strong$^*$-topology (and hence in the weak$^*$-topology) of $A^{**}$. By hypothesis $$0 = \{T(z_\lambda), b_n, 1\} + \{ z_{\lambda}, T(b_n),1\}, \hbox{ for all } \lambda.$$ Taking weak$^*$-limits in the above equality we get from the weak$^*$-continuity of $T\equiv T^{**}$ that $$0 = \{T (1-p_n), b_n, 1\} + \{ 1-p_n, T(b_n),1\}, \hbox{ for all } n,$$ which implies, via norm continuity, that $$0 = \{T (1-r(b)), b, 1\} + \{ 1-r(b), T(b),1\}= T (1-r(b)) \circ b +  (1-r(b)) \circ T(b)^*,$$ or equivalently $$ T (r(b)) \circ b = (1-r(b)) \circ T(b)^*.
$$ Since the range projection of every power $b^m$ with $m\in \mathbb{N}$ coincides with the $r(b)$ we can apply the above argument to deduce that $$T (r(b)) \circ b^m = (1-r(b)) \circ T(b^m)^*, \hbox{ for all natural } m,$$ and by linearity and norm continuity of the product we have \begin{equation}\label{eq 2 Jun2} T (r(b)) \circ z = (1-r(b)) \circ T(z)^*, \hbox{ for all } z = z^*\in A_b. \end{equation} A standard argument involving weak$^*$-continuity of $T^{**}\equiv T$ gives \begin{equation}\label{eq 1 June2} T (r(b)) \circ r(b) = (1-r(b)) \circ T(r(b))^*.
\end{equation}

Combining that $T^{**} \equiv T$ is a derivation with \eqref{eq 1 June2} we get  $$T(r(b)) = T (r(b))  r(b)  + r(b) T (r(b)) = (1-r(b))  T(r(b))^* +  T(r(b))^*  (1-r(b)). $$ By \cite[Proposition 3.7]{EssaPeRa16b}, we know that $r(b) T(r(b)) r(b)=0 = r(b) T(r(b))^* r(b)$, and by the equivalence $(f)\Leftrightarrow (a)$ in page \pageref{eq reformulations of g der} we have  $$(1-r(b)) T(r(b))^*  (1-r(b)) =0= (1-r(b)) T(r(b))  (1-r(b)),$$ and thus  $$r(b) T(r(b)) = r(b) T(r(b))^*  (1-r(b))$$ and $$(1-r(b)) T(r(b)) = (1-r(b))  T(r(b))^* .$$ Adding the last two identities we derive at $$T(r(b))^*= T(r(b)).$$ We have proved that $T(r)^* =T(r)$ for every range projection $r$ of a hermitian element in $A$\smallskip

We return to $A_b \cong C_0(\sigma(b))$. We observe that every projection of the form $p=\chi_{_{([-\|b\|, -\delta)\cup (\delta,\|b\|])\cap \sigma(b)}}\in C_0(\sigma(b))^{**}$, with $0<\delta<\|b\|$, is the range projection of an function in $C_0(\sigma(b))$. Furthermore, every projection of the form $q=\chi_{_{([-\theta, -\delta)\cup (\delta,\theta])\cap \sigma(b)}}\in C_0(\sigma(b))^{**}$ with $0<\delta<\theta<\|b\|$ can be written as the difference of two projections of the previous type. We have shown in the previous paragraph that $T(p)^*=T(p)$ for every projection $p$ of the first type, and consequently for every projection of the second type. Since $b$ can be approximated in norm by finite linear combinations of mutually orthogonal projections $q_j$ of the second type, and $T$ is continuous, we conclude that $T(b)^* = T(b)$, which finishes the proof.
\end{proof}

The conclusion after Corollary \ref{c bl triple derivable at zero with T1=0} is now clear.

\begin{corollary}\label{c T1 skew} Let $T: A\to A$ be a bounded linear map on a unital C$^*$-algebra $A$. Suppose $T$ is triple derivable at zero and $T(1)^*=-T(1)$. Then $T$ is a triple derivation.
\end{corollary}

\begin{proof} As in the proof of Corollary \ref{c derivable at 1 are triple derivations}. Since $T(1)^* =-T(1),$ the mapping $\delta(T(1),1):A \to A,$ $\delta(T(1),1) (x) = \{T(1),1,x\}- \{1,T(1),x\}$ is a triple derivation. Since the the linear combination of linear maps which are triple derivations at zero is a triple derivation at zero, the mapping $\tilde{T}=T-\frac{1}{2}\delta(T(1),1)$ is a triple derivation at zero, and $\tilde{T} (1) = 0$. Applying Corollary \ref{c bl triple derivable at zero with T1=0}, we derive that $\tilde{T}$ is a $^*$-derivation, and hence a triple derivation. Therefore $T= \tilde{T} +\frac{1}{2}\delta(T(1),1)$ is a triple derivation.
\end{proof}

The above results, are somehow optimal, in the sense that there exists bounded linear maps which are triple derivations at zero but they are not triple derivations. For example, let $Z$ be the center of a unital C$^*$-algebra $B$, and let us pick $b_0\in B$ with $0\neq b_0 \neq - b_0^*$. We define a bounded linear mapping $T: Z\to B$ by $T(x) = b_0 x$. Clearly, $T(1) = b_0\neq - b_0^*$ implies that $T$ is not a triple derivation (cf. Lemma \ref{lemma33}\cite[proof of Lemma 1]{HoMarPeRu}, \cite[Lemma 3.4]{HoPeRu} or \cite[Lemma 2.1]{KuOikPeRu14}). Suppose $\{ a,b,c \} =0$ in $Z$. Since $Z$ is the center of $B$, $b_0$ commutes with every element in $B$, and then we have $$\{T(a),b,c\} +\{a,T(b),c\} + \{a,b,T(c)\}= \{b_0 a,b,c\} +\{a,b_0 b,c\} + \{a,b,b_0 c\}$$ $$  =2 b_0 \{a,b,c\} + b_0^* \{a,b,c\} =0= T(0) , $$ witnessing that $T$ is a triple derivation at zero.\smallskip

With the help of \cite[Theorem 2.8 and Proposition 2.4]{EssaPeRa16} we can now throw some new light about the automatic continuity of generalized derivations and linear maps which are derivable at zero on a von Neumann algebra.

\begin{theorem}\label{t automatic cont gen der and triple der at zero on von Neumann} Let $T: M\to M$ be a linear mapping on a von Neumann algebra. Suppose that for each $a,b,c$ in a commutative von Neumann subalgebra $\mathcal{B}$ with $a b = bc=0$ we have $a T(b) c= 0$. Then $T$ is continuous.
\end{theorem}

\begin{proof}
We can mimic the ideas in the proof of \cite[Theorem 2.8]{EssaPeRa16}. Assume that $T$ satisfies hypothesis $(a)$ (respectively $(b)$). Let $\mathcal{B}$ be a commutative von Neumann subalgebra of $M$ containing the unit of $M$. Fix $a,b\in \mathcal{B}$ with $a b=0$, and define a linear mapping $L_{a,b} : \mathcal{B}\to M$, by $L_{a,b} (x) = aT(bx)$. Let $c,d$ be elements in $\mathcal{B}$ with $cd=0$, then, by hypothesis $L_{a,b} (c) d = aT(bc)d =0$ (because $a(bc) d=0$). This proves that $L_{a,b}$ is a linear left-annihilator preserving. Proposition 2.4 in \cite{EssaPeRa16} assures that $L_{a,b}$ is continuous and a left-multiplier, that is, $L_{a,b} (x) = L_{a,b} (1) x,$ for every $x\in \mathcal{B}$. This property assures that, for each $x\in \mathcal{B}$ the mapping $R_x : \mathcal{B}\to M$, $R_{x} (z) = T(xz)-T(z)x$ satisfies $a R_{x} (b) = 0,$ for every $ab=0$ in $\mathcal{B}$. Consequently, $R_{x}$ is a linear right-annihilator preserving, and Proposition 2.4 in \cite{EssaPeRa16} proves that $R_{x}$ is a continuous right multiplier. We have shown that $$T(y x) -T(y) x = R_x (y) = y R_x (1) = y T(x),$$ for every $x,y\in \mathcal{B}$, or equivalently, $T|_{\mathcal{B}}: \mathcal{B}\to M$ is a derivation. Theorem 2 in \cite{Ringrose72} assures that $T|_{\mathcal{B}}$ is a bounded linear map, and this holds for every abelian von Neumann subalgebra  $\mathcal{B}$ of $M$ containing the unit of $M$. The continuity of $T$ follows as a consequence of \cite[Theorem 2.5]{Ringrose74}.
\end{proof}

Every generalized derivation on a von Neumann algebras satisfies the hypothesis of the above Theorem \ref{t automatic cont gen der and triple der at zero on von Neumann}. Surprisingly, the linear maps on a von Neumann algebra which are triple derivable at zero also satisfy the same hypothesis.

\begin{corollary}\label{c automatic cont gen der and triple der at zero on von Neumann} Every generalized derivation on a von Neumann algebra is continuous. Every linear map on a von Neumann algebra which is a derivation {\rm(}respectively, a triple derivation{\rm)} at zero is continuous.
\end{corollary}

\begin{proof} The first statement is clear. The statement concerning (associative) derivations at zero is a consequence of Lemma \ref{lemma1.2 associative der}. In order to prove the remaining statement, we assume that $T:M\to M$ is a linear map on a von Neumann algebra which is triple derivable at zero. Let $\mathcal{B}$ be a commutative von Neumann subalgebra of $M$ containing the unit, and let us take $a,b,c\in \mathcal{B}$ with $a b= bc=0$. By the commutativity of $\mathcal{B}$ we have $ b\perp a, a^*$ and $b\perp c,c^*$. By Lemma \ref{lemma1.2} we have $0=\{ a^*, T(b), c^*\} = \frac12(a^* T(b)^* c^* + c^* T(b)^* a^* ),$ and then $a T(b) c + c T(b) a=0$.\smallskip

Since $\mathcal{B}$ is a commutative von Neumann algebra, the every element in $\mathcal{B}$ is normal and its left and right range projections coincide and belong to $\mathcal{B}$. Let $r(b)\in \mathcal{B}$ be the range projection of $b$. Since $0=b (1-r(b)) = (1-r(b))b$, we deduce from what is proved in the first paragraph that $(1-r(b))  T(b) (1-r(b)) =0$. Finally, since $a = a(1-r(b)) $ and $c= c(1-r(b))$, we have $a T(b) c = a (1-r(b)) T(b) (1-r(b)) c=0$. This shows that $T$ satisfies the hypothesis in Theorem \ref{t automatic cont gen der and triple der at zero on von Neumann}, and thus $T$ is continuous.
\end{proof}

Combining Corollary \ref{c automatic cont gen der and triple der at zero on von Neumann} with Theorem \ref{t continuous triple der at zero are g der} and Corollaries \ref{c bl triple derivable at zero with T1=0} and \ref{c T1 skew} we establish the following.

\begin{corollary}\label{c final for vN algebras} Let $M$ be a von Neumann algebra. Suppose $T: M\to M$ is a linear map which is a triple derivation at zero. Then the following statements hold:
\begin{enumerate}[$(a)$]\item $T$ is a continuous generalized derivation;
\item If $T(1)=0$ then $T$ is a {\rm(}continuous{\rm)} $^*$-derivation and a triple derivation;
\item If $T(1)^*=-T(1)$ then $T$ is a {\rm(}continuous{\rm)} triple derivation. $\hfill\Box$
\end{enumerate}
\end{corollary}

We can also apply the above results to relax some of the hypothesis in previous papers. We begin with a version of the results reviewed in page \pageref{eq reformulations of g der}  for non-necessarily continuous linear maps.

\begin{corollary}\label{c non continuous generalized derivation} Let $T: M\to M$ be a linear map on a von Neumann algebra. The the following statements are equivalent:
\begin{enumerate}[$(a)$]
\item $T$ is a generalized derivation;
\item $a T(b) c = 0$, whenever $ab=bc=0$ in $M$;
\item $a T(b) c = 0$, whenever $ab=bc=0$ in $M_{sa}$;
\item[$(d)$] $a T(b) c + c T(b) a = 0$, whenever $ab=bc=0$ in $M_{sa}$;
\item[$(e)$] $a T(b) a = 0$, whenever $ab=0$ in $M_{sa}$.
\item[$(f)$] For each $b$ in $M_{sa}$ we have $(1-r(b)) T(b) (1-r(b))=0$, where $r(b)$ denotes the range projection of $b$ in $M$.
\end{enumerate}
\end{corollary}

\begin{proof} The implication $(a)\Rightarrow (b)$ follows from Lemma \ref{lemma1.2 associative der}, while the implications $(b)\Rightarrow (c)\Rightarrow (d) \Rightarrow (e)$ are clear.\smallskip

If we show that any linear mapping $T:M\to M$ satisfying $(e)$ or $(f)$ is continuous then the remaining implications will follow from the arguments given in page \pageref{eq reformulations of g der}. If $T$ satisfies $(e)$ then $T$ also satisfies $(f)$ because $r(b)\in M_{sa}$ for every $b\in M_{sa}$ and $(1-r(b)) b=0$.\smallskip

Let $\mathcal{B}$ be a commutative von Neumann subalgebra of $M$, and let us take $x,y,z\in \mathcal{B}$ with $x y = y z=0$. We can write $x = x_1 + i x_2,$ $y = y_1 + i y_2$, and $z = z_1 + i z_2$ with $x_j,y_j, z_j\in M_{sa}$.\smallskip

Suppose $T$ satisfies $(e)$. Since $\mathcal{B}$ is commutative $x_j y_k = y_k x_j=0$ and $z_j y_k = y_k z_j=0$ for all $j,k=1,2$. Clearly $x_j = x_j (1-r(y_k))$ and $y_j = y_j (1-r(y_k))$,  for all $j,k=1,2$, which assures that $x = x (1-r(y_k))$ and $y = y (1-r(y_k))$ for all $k=1,2$. Therefore, by assumptions, we obtain $$ x T(y) z= x T(y_1) z + i x T(y_2) z = x (1-r(y_1)) T(y_1) (1-r(y_1)) z $$ $$+ i x (1-r(y_2)) T(y_2) (1-r(y_2)) z =0,$$ which finishes the proof.
\end{proof}

In \cite[Theorem 4]{JingLuLi2002} W. Jing, S.J. Lu, and P.T. Li prove that a continuous linear map $\delta$ on a von Neumann algebra is a generalized derivation whenever it is a derivation at zero. If additionally $\delta (1)=0$, then $\delta$ is a derivation. Corollary \ref{c automatic cont gen der and triple der at zero on von Neumann} assures that the hypothesis concerning the continuity of $\delta$ can be relaxed.

\begin{corollary}\label{c JingLuLi without continuity} Let $\delta$ be a linear map on a von Neumann algebra. Suppose $\delta$ is a derivation at zero. Then $\delta$ is a (continuous) generalized derivation. If additionally $\delta (1)=0$, then $\delta$ is a derivation.$\hfill\Box$
\end{corollary}

\section{Triple homomorphisms at a fixed point}\label{sec: triple hom at a point}

Let $A$ and $B$ be C$^*$-algebras. According to the notation in \cite{BurCabSanPe2016}, a linear map $T : A \to B$ is said to be a \emph{$^*$-homomorphism at $z\in A$} if $$\hbox{$a b^* =z\Rightarrow T(a) T(b)^*=T(z),$ and $c^* d =z\Rightarrow T(c)^* T(d)=T(z)$.}$$ In \cite[Theorem 2.5]{BurCabSanPe2016} it is shown that when $A$ is unital, a linear map $T : A \to B$ which is a $^*$-homomorphism at $1$ is continuous and a Jordan $^*$-homomorphism. The same conclusion hold if there exists a non-zero projection $p\in A$ such that $T$ is a $^*$-homomorphism at $p$ and at $1-p$ \cite[Corollary 2.12]{BurCabSanPe2016}. Furthermore, in the above setting, $T$ is a $^*$-homomorphism if and only if $T$ is a $^*$-homomorphism at $0$ and at $1$ \cite[Corollary 2.11]{BurCabSanPe2016}. If $A$ is assumed to be simple and infinite, then a linear map $T: A\to B$ is a $^*$-homomorphism if and only if $T$ is a $^*$-homomorphism at the unit of $A$ (cf. \cite[Theorem 2.8]{BurCabSanPe2016}). In the just quoted paper, it also studied when a continuous linear map which is a $^*$-homomorphism at a unitary element is a Jordan $^*$-homomorphism.\smallskip

We recall some terminology needed in forthcoming results. For each partial isometry $e$ in a C$^*$-algebra $A$ (i.e., $ee^*e =e$), we can decompose $A$ as a direct sum of the form $$A= (ee^* A e^*e) \oplus \left((1-ee^*) A  e^*e \oplus ee^* A (1-e^*e)\right)\oplus (1-ee^*) A (1-e^*e).$$ The above decomposition is called the \emph{Peirce decomposition} of $A$ associated with $e$. The subsets $A_2(e) = ee^* A e^*e$, $A_1(e) = (1-ee^*) A  e^*e \oplus ee^* A (1-e^*e),$ and  $A_0(e) = (1-ee^*) A (1-e^*e)$ are called the \emph{Peirce subspaces} associated with $e$.

\subsection{Triple homomorphisms at the unit element}\  \vspace{1mm}

We explore first the behavior on the projections of a linear map which is a triple homomorphism at the unit.

\begin{lemma}\label{lemma31} Let $T:A\to B$ be a linear map between C$^*$-algebras, where $A$ is unital. Suppose $T$ is a triple homomorphism at the unit of $A.$ Then the following statements hold:\begin{enumerate}[$(a)$]
\item $T(1)$ is a partial isometry;
\item The identity $\{T(p),T(1),T(p)\}=\{T(p),T(1),T(1)\}$ holds for every projection $p\in A$.
\end{enumerate}
\end{lemma}

\begin{proof} $(a)$ The identity $\{1,1,1\}=1$ and the hypothesis imply $$T(1)=T(\{1,1,1\})=\{T(1),T(1),T(1)\}=T(1)T(1)^*T(1).$$

$(b)$ Let $p\in A$ be a projection. We know that $\{(1-2p),1,(1-2p)\}=1$. Thus
$$T(1)=T(\{(1-2p),1,(1-2p)\})=\{T(1-2p),T(1),T(1-2p)\}$$
$$= \{T(1),T(1),T(1)\} -4 \{T(p),T(1),T(1)\} + 4 \{T(p),T(1),T(p)\},$$ which combined with $(a)$ gives the desired statement.
\end{proof}

For the next result we shall assume continuity of our linear map.

\begin{proposition}\label{p symmetry}
Let $T:A\to B$ be a continuous linear map between C$^*$-algebras, where $A$ is unital. Suppose $T$ is a triple homomorphism at the unit of $A.$ Then the identity
$$\{T(x^*),T(1),T(1)\}=\{T(1),T(x),T(1)\},$$ holds for all $x\in A$. Consequently, $T(1)$ is a partial isometry and  $T(A)\subseteq B_2 (T(1))\oplus B_0(T(1))$.
\end{proposition}

\begin{proof} Let us take $a\in A_{sa}$ and $t\in \RR.$ Since $e^{ita}$ is a unitary in $A$, we have $1=\{e^{iat},e^{iat},1\},$ and by the assumptions we get $T(1)=\{T(e^{iat}),T(e^{iat}),1)\}.$ By taking derivative in $t=0$, we obtain
$$0=\{T(a),T(1),T(1)\}-\{T(1),T(a),T(1)\}.$$ Since for
$x\in A,$ we can write $x=a+ib$ with $a,\; b\in  A_{sa},$ it follows from the above that
$$\{T(x^*),T(1),T(1)\}=\{T(a),T(1),T(1)\}-i\{T(b),T(1),T(1)\}$$
$$=\{T(1),T(a),T(1)\}-i\{T(1),T(b),T(1)\}=\{T(1),T(x),T(1)\}.$$

It follows from Lemma \ref{lemma31} that $e=T(1)$ is a partial isometry. For the final statement we observe that for each $a\in A_{sa}$ we have $$ee^* T(a) e^*e + \frac12 ( (1-ee^*) T(a)  e^*e \oplus ee^* T(a) (1-e^*e) ) = \{T(a),T(1),T(1)\}$$ $$=\{T(1),T(a),T(1)\} = e T(a)^* e = ee^* (e T(a)^* e) e^* e, $$ which shows that $(1-ee^*) T(a)  e^*e \oplus ee^* T(a) (1-e^*e) = 0,$ and hence $$T(a) =  ee^* T(a) e^*e + (1-ee^*) T(a) (1-e^*e)\in B_2(e) \oplus B_0(e).$$
\end{proof}

For the next result we explore new arguments with higher derivatives.

\begin{proposition}\label{p zero annihilating the zero part}
Let $T:A\to B$ be a continuous linear map between C$^*$-algebras, where $A$ is unital. Suppose $T$ is a triple homomorphism at the unit of $A.$ Then $T(1)$ is a partial isometry and $T(A)\subseteq B_2 (T(1))$.
\end{proposition}

\begin{proof} As in previous cases, we fix $a\in A_{sa}$. Since $1=\{e^{iat},e^{2 iat},e^{iat}\},$ for all $t\in \RR,$ by hypothesis we get $T(1)=\{T(e^{iat}),T(e^{2iat}),T(e^{iat})\}.$ By taking a first derivative in $t$, we obtain
$$0=2 \{T( a e^{iat}),T(e^{2iat}),T(e^{iat})\} -2 \{T(e^{iat}),T(a e^{2iat}),T(e^{iat})\},$$ for all $t\in \mathbb{R}$. By taking subsequent derivatives at $t$ we get
\begin{equation}\label{eq second derivative} 0=2 \{T( a^2 e^{iat}),T(e^{2iat}),T(e^{iat})\} - 8 \{T( a e^{iat}),T(a e^{2iat}),T(e^{iat})\}
\end{equation} $$ + 2 \{T( a e^{iat}),T(e^{2iat}),T(a e^{iat})\}+4 \{T(e^{iat}),T(a^2 e^{2iat}),T(e^{iat})\}, $$ and
$$0=2 \{T( a^3 e^{iat}),T(e^{2iat}),T(e^{iat})\} - 4 \{T( a^2 e^{iat}),T(a e^{2iat}),T(e^{iat})\}$$ $$+ 2 \{T( a^2 e^{iat}),T(e^{2iat}),T(a e^{iat})\} - 8 \{T( a^2 e^{iat}),T(a e^{2iat}),T(e^{iat})\}$$ $$ +16 \{T( a e^{iat}),T(a^2 e^{2iat}),T(e^{iat})\} - 8 \{T( a e^{iat}),T(a e^{2iat}),T(a e^{iat})\}$$ $$+ 2 \{T( a^2 e^{iat}),T(e^{2iat}),T(a e^{iat})\}-4 \{T( a e^{iat}),T(a e^{2iat}),T(a e^{iat})\}$$ $$+ 2 \{T( a e^{iat}),T(e^{2iat}),T(a^2 e^{iat})\} +4 \{T(a e^{iat}),T(a^2 e^{2iat}),T(e^{iat})\}$$ $$-8 \{T(e^{iat}),T(a^3 e^{2iat}),T(e^{iat})\} +4 \{T(e^{iat}),T(a^2 e^{2iat}),T(a e^{iat})\}.$$ By replacing $t$ with $0$ we get
\begin{equation}\label{eq TaTaTa} 0=2 \{T( a^3),T(1),T(1)\} - 12 \{T( a^2 ),T(a ),T(1)\}+ 6 \{T( a^2 ),T(1),T(a )\}
\end{equation} $$-8 \{T(1),T(a^3 ),T(1)\} +24 \{T( a),T(a^2 ),T(1)\} -12 \{T( a),T(a),T(a )\}.$$

Now, by Proposition \ref{p symmetry}, we write $T(a) = x_2 + x_0$, $T(a^2) =y_2+y_0$ and $T(a^3) = z_2 +z_0$, where $x_j,y_j,z_j\in B_j(T(1))$ for all $j=2,0$. It is not hard to check that $\{T( a^3),T(1),T(1)\},$ $\{T( a^2 ),T(a ),T(1)\},$ $\{T( a^2 ),T(1),T(a )\},$ $\{T(1),T(a^3 ),T(1)\}$, and $\{T( a),T(a^2 ),T(1)\}$ all lie in $B_2(T(1)),$ while\linebreak $\{T( a),T(a),T(a )\} = \{x_2,x_2,x_2\}+ \{x_0,x_0,x_0\}$ with $\{x_2,x_2,x_2\}\in B_2(T(1))$ and $\{x_0,x_0,x_0\}\in B_0(T(1))$. It follows from \eqref{eq TaTaTa} that $\{x_0,x_0,x_0\}= x_0x_0^* x_0=0$ and hence $x_0 x_0^*x_0 x_0^*=x_0=0$. We have therefore shown that $T(a)\in B_2(T(1))$ for all $a\in A_{sa}$. The desired conclusion follows from the linearity of $T$.
\end{proof}

We can now establish our main result for bounded linear maps which are triple homomorphisms at the unit element. We recall that given a partial isometry $e$ in a C$^*$-algebra $A$, the Peirce subspace $A_2(e) = ee^* A e^*e$ is a JB$^*$-algebra with Jordan product $x\circ_e y := \{ x, e, y\}= \frac12 (x e^* y + y e^* x),$ and involution $x^{\sharp_e} =\{e,x,e\} = e x^*e$.

\begin{theorem}\label{theorem3131}
Let $T:A\to B$ be a continuous linear map between C$^*$-algebras, where $A$ is unital. Suppose $T$ is a triple homomorphism at the unit of $A.$ Then $T$ is a triple homomorphism. Furthermore, $T(1)$ is a partial isometry and $T: A \to B_2 (T(1))$ is a Jordan $^*$-homomorphism.
\end{theorem}

\begin{proof} By Lemma \ref{lemma31} the element $T(1)$ is a partial isometry. Proposition \ref{p zero annihilating the zero part} proves that $T(A)\subseteq B_2 (T(1))$, and consequently, Proposition \ref{p symmetry} guarantees that $T(x^*) =\{T(x^*),T(1),T(1)\}= \{T(1), T(x), T(1)\} = T(x)^{\sharp_{T(1)}}$ for every $x\in A$. It is not hard to see from these properties that $T(a^2) =\{T( a^2 ),T(1),T(1)\} =\{T( a ),T(a ),T(1)\}=\{T( a ),T(1),T(a)\},$ for every $a\in A_{sa}$.\smallskip

The proof will be completed if we show that $T: A \to B_2 (T(1))$ is a Jordan $^*$-homomorphism. We shall only prove that $T$ preserves the corresponding Jordan product. Following the arguments in the proof of Proposition \ref{p zero annihilating the zero part}, and replacing $t$ with $0$ in \eqref{eq second derivative} we arrive at $$ 0=2 \{T( a^2 ),T(1),T(1)\} - 8 \{T( a ),T(a ),T(1)\} $$ $$+ 2 \{T( a ),T(1),T(a)\}+4 \{T(1),T(a^2),T(1)\}, $$ and then $$ T(a^2) =  \{T( a^2 ),T(1),T(1)\} = \{T( a ),T(a ),T(1)\} $$ $$= \{T( a ),T(1),T(a )\} = T(a)\circ_{T(1)} T(a),$$ for all $a\in A_{sa}$. A standard polarization argument proves that $T$ is a Jordan $^*$-homomorphism.\end{proof}

\begin{problem}
Let $T:A\to B$ be a linear map between C$^*$-algebras, where $A$ is unital. Suppose $T$ is a triple homomorphism at the unit of $A.$ Is $T$ continuous?
\end{problem}

Accordingly to the structure of this note, the reader is probably interested on bounded linear maps which are triple homomorphisms at zero. It is not a big surprise that these maps are directly connected with the so-called \emph{orthogonality preserving operators} in the sense studied, for example, in \cite{Wolff94,BurFerGarMarPe08}, and subsequent papers. We recall that a linear map $T$ between C$^*$-algebras is called orthogonality preserving if the equivalence $$a\perp b \hbox{ in } A \Rightarrow T(a)\perp T(b) \hbox{ in } B.$$
It is known that elements $a,b$ in a C$^*$-algebra $A$ are orthogonal if, and only if, $\{a,a,b\}=0$ (see, for example, \cite[Lemma 1 and comments in page 221]{BurFerGarMarPe08}).  The main result in \cite{BurFerGarMarPe08} establishes a complete description of those continuous linear maps between C$^*$-algebras which preserver orthogonal elements. Let $T :A \to B$ be a bounded linear map between two C$^*$-algebras, Corollary 18 in \cite{BurFerGarMarPe08} proves that $T$ is orthogonality preserving if, and only if, $T$ preserves zero-triple-products (i.e. $\{a,b,c\}=0$ in $A$ implies $\{T(a),T(b),T(c)\}=0$ in $B$), and the latter is precisely the notion of being a triple homomorphism at zero.


\begin{thebibliography}{22}
\bibitem{AlBreExVill09} J. Alaminos, M. Bresar, J. Extremera, A. Villena, Maps preserving zero products, \emph{Studia Math.} \textbf{193}, no. 2, 131-159 (2009).

\bibitem{AyuKudPe2014} S. Ayupov, K. Kudaybergenov, A.M. Peralta, A survey on local and 2-local derivations on C$^*$- and von Neuman algebras, \emph{Contemporary Mathematics, Amer. Math. Soc.} Volume 672, 73-126 (2016).

\bibitem{BleNe} D.P. Blecher, M. Neal, {Open partial isometries and positivity
in operator spaces}, \emph{Studia Math.} \textbf{182}, no. 3, 227-262 (2007).

\bibitem{Bre07} M. Bre\v{s}ar, Characterizing homomorphisms, derivations and multipliers in rings with idempotents, \emph{Proc. Roy. Soc. Edinburgh
Sect. A} \textbf{137} 9-21 (2007).

\bibitem{BurCabSanPe2016} M. Burgos, J. Cabello-S{\'a}nchez, A.M. Peralta, Linear maps between C$^*$-algebras that are $^*$-homomorphisms at a fixed point, preprint 2016. arXiv:1609.07776

\bibitem{BurFerGarMarPe08} M. Burgos, F.J. Fern{\' a}ndez-Polo, J.J Garc{\'e}s, J.M Moreno, A.M. Peralta, Orthogonality preservers in C$^*$-algebras, JB$^*$-algebras and JB$^*$-triples, \emph{J. Math. Anal. Appl.} \textbf{348}(1), 220-233 (2008).

\bibitem{BurFerGarPe2014} M. Burgos, F.J. Fern{\' a}ndez-Polo, J.J. Garc{\'e}s, and A.M. Peralta, \textit{Local triple derivations on C$^*$-algebras}, \emph{Comm. Algebra} \textbf{42}, no. 3, 1276--1286  (2014).

\bibitem{BurFerPe2014} M. Burgos, F.J. Fern{\' a}ndez-Polo, A.M. Peralta, Local triple derivations on C$^*$-algebras and JB$^*$-triples, \emph{Bull. London Math. Soc.}, \textbf{46} (4), 709-724 (2014).

\bibitem{ChebKeLee} M. A. Chebotar, W.-F. Ke and P.-H. Lee. Maps characterized by action on zero products, \emph{Pacific J. Math.} \textbf{216}, 217-228 (2004).

\bibitem{EssaPeRa16} A.B.A. Essaleh, A.M. Peralta, M.I. Ram{\'i}rez, Weak-local derivations and homomorphisms on C$^*$-algebras, \emph{Linear and Multilinear Algebra} \textbf{64}, No. 2 169-186 (2016).

\bibitem{EssaPeRa16b} A.B.A. Essaleh, A.M. Peralta, M.I. Ram{\'i}rez, CORRIGENDUM: Weak-local derivations and homomorphisms on C$^*$-algebras, \emph{Linear and Multilinear Algebra} \textbf{64}, No. 5, 1009-1010 (2016).

\bibitem{Hanche} H. Hanche-Olsen and E. St{\o }rmer, \textit{Jordan operator algebras, Monographs and Studies in Mathematics 21}, Pitman, London-Boston-Melbourne 1984.

\bibitem{Harr74} L.A. Harris, Bounded symmetric homogeneous domains in infinite dimensional spaces. \emph{In Proceedings on Infinite Dimensional Holomorphy} (Internat. Conf., Univ. Kentucky, Lexington, Ky., 1973), pp. 13-40. Lecture Notes in Math., Vol. 364, Springer, Berlin, 1974.

\bibitem{Harris81} L.A. Harris, {A generalization of C$^*$-algebras}, \emph{Proc. London Math. Soc.} (3) \textbf{42}, no. 2, 331-361 (1981).

\bibitem{HoMarPeRu} T. Ho, J. Martinez-Moreno, A.M. Peralta, B. Russo,
 Derivations on real and complex JB$^\ast$-triples, \emph{J. London Math. Soc.} (2)
 \textbf{65}, no. 1, 85-102 (2002).

\bibitem{HoPeRu} T. Ho, A.M. Peralta, B. Russo, Ternary Weakly Amenable C$^*$-algebras and JB$^*$-triples, \emph{Quart. J. Math. (Oxford)} \textbf{64}, no. 4, 1109-1139 (2013).

\bibitem{Houqi} J.C. Hou, X.F. Qi, Additive maps derivable at some points on $J$-subspace lattice algebras, \emph{Linear Algebra Appl.} \textbf{429}, 1851-1863 (2008).

\bibitem{Jing} W. Jing, On Jordan all-derivable points of $B(H)$, \emph{Linear Algebra Appl.} \textbf{430} (4), 941-946 (2009).

\bibitem{JingLuLi2002} W. Jing, S.J. Lu, P.T. Li, Characterization of derivations on some operator algebras, \emph{Bull. Austral. Math. Soc.} \textbf{66}, 227-232 (2002).





\bibitem{John96} B.E. Johnson, Symmetric amenability and the nonexistence of Lie and Jordan derivations, \emph{Math. Proc. Cambridge Philos. Soc.}
\textbf{120}, no. 3, 455-473 (1996).


\bibitem{Kaup83} W. Kaup, {A Riemann Mapping Theorem for bounded symmentric domains in complex Banach spaces}, \emph{Math. Z.} \textbf{183}, 503--529 (1983).



\bibitem{KuOikPeRu14} K. Kudaybergenov, T. Oikhberg, A. M. Peralta, and B. Russo, 2-local triple derivations on
von Neumann algebras, \emph{Illinois J. Math.} \textbf{58}, no. 4, 1055-1069 (2014).

\bibitem{LiPan} J. Li, Z. Pan, Annihilator-preserving maps, multipliers, and derivations, \emph{Linear Algebra Appl.} \textbf{423}, 5-13 (2010).

\bibitem{Lu2009} F. Lu, Characterizations of derivations and Jordan derivations on Banach algebras,\emph{
Linear Algebra Appl.} \textbf{430}, No. 8-9, 2233-2239 (2009).

\bibitem{NealRusso} M. Neal, B. Russo, \textit{Operator space characterizations of C$^*$-algebras and ternary rings}, {Pacific J. Math.} \textbf{209} (2003), no. 2, 339-364.

\bibitem{Ped} G.K. Pedersen, \emph{C$^*$-algebras and their automorphism groups},
London Mathematical Society Monographs Vol. 14, Academic Press, London, 1979.


\bibitem{Ringrose72} J.R. Ringrose, {Automatic continuity of derivations of operator algebras},
\emph{J. London Math. Soc.} (2) \textbf{5} , 432-438 (1972).

\bibitem{Ringrose74} J.R. Ringrose, Linear functionals on operator algebras and their Abelian subalgebras, \emph{J. London Math. Soc.} (2) \textbf{7}, 553-560 (1974).

\bibitem{S} S. Sakai, \emph{C$^*$-algebras and $W^*$-algebras}. Springer Verlag. Berlin 1971.


\bibitem{Tak} M. Takesaki, \emph{Theory of operator algebras I}, Springer-Verlag, Berlin 1979.

\bibitem{Wolff94} M. Wolff, Disjointness preserving operators in C$^*$-algebras,
\emph{Arch. Math.} \textbf{62}, 248-253 (1994). MR1259840
(94k:46122).



%

\bibitem{ZhangHouQi2014} Y.F. Zhang, J.C. Hou, X.F. Qi, Characterizing derivations for any nest algebras on Banach space by their behaviors at an injective operator, \emph{Linear Algebra Appl.}, \textbf{449} (15), 312-333 (2014).

\bibitem{ZhangHouQi2014b} Y.F. Zhang, J.C. Hou, X.F. Qi, All-derivable subsets for nest algebras on Banach spaces, \emph{Int. Math. Forum} \textbf{9}, no. 1-4, 1-11  (2014).

\bibitem{Zhu2007} J. Zhu, All-derivable points of operator algebras, \emph{Linear Algebra Appl.} \textbf{427}, 1-5 (2007).

\bibitem{ZhuXio2002} J. Zhu, C.P. Xiong, Generalized derivable mappings at zero point on nest algebras, \emph{Acta Math. Sinica} \textbf{45}, 783-788 (2002).

\bibitem{ZhuXio2005} J. Zhu, C.P. Xiong, Generalized derivable mappings at zero point on some reflexive operator algebras, \emph{Linear Algebra Appl.} \textbf{397}, 367-379 (2005).

\bibitem{ZhuXio2007} J. Zhu, C.P. Xiong, Derivable mappings at unit operator on nest algebras, \emph{Linear Algebra Appl.} \textbf{422}, 721-735 (2007).

\bibitem{ZhuXion2008} J. Zhu, C.P. Xiong, All-derivable points in continuous nest algebras, \emph{J. Math. Anal. Appl.} \textbf{340}, 845-853 (2008).

\bibitem{ZhuXiLi} J. Zhu, Ch. Xiong, P. Li, Characterizations of all-derivable points in $B(H)$, \emph{Linear Multilinear Algebra} \textbf{64}, no. 8, 1461-1473 (2016).

\bibitem{ZhuZhao2013} J. Zhu, S. Zhao, Characterizations all-derivable points in nest algebras, \emph{Proc. Amer. Math. Soc.} \textbf{141} (7), 2343-2350  (2013).

\end{thebibliography}
\end{document}